\documentclass[10pt,reqno]{amsart}
\usepackage{bbm}
\usepackage{mathrsfs}
\usepackage{amsthm}
\usepackage{amsfonts} 
\usepackage[dvipsnames,usenames]{color}
\usepackage{graphicx,latexsym,bm,amsmath,amssymb,verbatim,multicol,lscape}
\newtheorem{thm}{Theorem} [section]
\newtheorem{cor}[thm]{Corollary}
\newtheorem{cnj}[thm]{Conjecture}
\newtheorem{lem}[thm]{Lemma}

\theoremstyle{definition}

\theoremstyle{remark}

\numberwithin{equation}{section}

\begin{document}
\title[On the $p$-adic properties of Stirling numbers of the first kind ]
{On the $p$-adic properties of Stirling numbers of the first kind }
\begin{abstract}
Let $n, k$ and $a$ be positive integers. The Stirling numbers of the first
kind, denoted by $s(n,k)$, count the number of permutations of $n$
elements with $k$ disjoint cycles. Let $p$ be a prime. In recent years,
Lengyel, Komatsu and Young, Leonetti and Sanna, Adelberg, Hong and Qiu
made some progress in the study of the $p$-adic valuations
of $s(n,k)$. In this paper, by using Washington's congruence on
the generalized harmonic number and the $n$-th Bernoulli number $B_n$
and the properties of $m$-th Stirling numbers of the first kind
obtained recently by the authors, we arrive at an exact
expression or a lower bound of $v_p(s(ap, k))$ with $a$ and
$k$ being integers such that $1\le a\le p-1$ and $1\le k\le ap$.
This infers that for any regular prime $p\ge 7$ and for arbitrary
integers $a$ and $k$ with $5\le a\le p-1$ and $a-2\le k\le ap-1$,
one has $v_p(H(ap-1,k)) < -\frac{\log{(ap-1)}}{2\log p}$ with
$H(ap-1, k)$ being the $k$-th elementary symmetric function
of $1, \frac{1}{2}, ..., \frac{1}{ap-1}$. This gives a partial
support to a conjecture of Leonetti and Sanna raised in 2017.
We also present results
on $v_p(s(ap^n,ap^n-k))$ from which one can derive that under
certain condition, for any prime $p\ge 5$, any odd number
$k\ge 3$ and any sufficiently large integer $n$, if $(a,p)=1$,
then $v_p(s(ap^{n+1},ap^{n+1}-k))=v_p(s(ap^n,ap^n-k))+2$.
It confirms partially Lengyel's conjecture proposed in 2015.
\end{abstract}
\author[S.F. Hong]{Shaofang Hong}
\address{Mathematical College, Sichuan University, Chengdu 610064, P.R. China}
\email{sfhong@scu.edu.cn; s-f.hong@tom.com; hongsf02@yahoo.com}
\author[M. Qiu]{Min Qiu$^*$}
\address{Mathematical College, Sichuan University, Chengdu 610064, P.R. China}
\address{School of Science, Xihua University, Chengdu 610039, P.R. China}
\email{qiumin126@126.com}
\thanks{$^*$ M. Qiu is the corresponding author. S.F. Hong was supported
partially by National Science Foundation of China Grant \#11771304.}
\maketitle

\section{Introduction}
Let $n$ and $k$ be positive integers such that $k\le n$.
The \emph{Stirling number of the first kind}, denoted by $s(n,k)$,
counts the number of permutations of $n$ elements with $k$ disjoint cycles.
One can also characterize $s(n,k)$ by
$$ (x)_n=x(x+1)(x+2)\cdots(x+n-1)=\sum_{k=0}^{n}s(n,k)x^k. $$
The \emph{Stirling number of the second kind}
$S(n,k)$ is defined as the number of ways to partition a set of
$n$ elements into exactly $k$ nonempty subsets. Thus
$$ S(n,k)=\frac{1}{k!}\sum_{i=0}^{k}(-1)^i\binom{k}{i}(k-i)^n.$$
There are many authors who are interested in the divisibility
properties of Stirling numbers of the second kind, see, for example,
\cite{[Ad],[HZ1],[TL],[TL2],[Lu],[Mi],[HZ2],[HZ3]}.
But very few seems to be known about the divisibility properties of $s(n,k)$.

The Stirling number of the first kind $s(n, k)$ is closely related to the
$k$-th elementary symmetric function $H(n, k)$ of $1,1/2,...,1/n$
by the following identity

$$
s(n+1,k+1)=n!H(n,k),
$$
(see Lemma 1.1 in \cite{[LS]}), where $H(n,0):= 1$ and
$$
H(n,k) := \sum_{1 \le i_1 < \cdots < i_k \le n} \frac{ 1 } { i_1 \cdots i_k }.
$$
For any positive integer $n$, let $v_p(n)$ represent the
{\it $p$-adic valuation} of $n$, i.e., $v_p(n)$ is the
biggest nonnegative integer $r$ with $p^r$ dividing $n$.
If $x=\frac{n_1}{n_2}$, where $n_1$ and $n_2$ are integers
and $n_2\ne 0$, then we define $v_p(x):=v_p(n_1)-v_p(n_2)$.
The Legendre formula about the $p$-adic valuation
of the factorial tells us that
$$
v_p(n!)=\frac{n-d_p(n)}{p-1},
$$
where $d_p(n)$ stands for the base $p$ digital sum of $n$.
Hence the investigation of $v_p(s(n+1,k+1))$ is equivalent
to that of $v_p(H(n,k))$.

It is known that $H(n,1)=H_n$ is the $n$-th harmonic number.
Theisinger \cite{[T]} and Nagell \cite{[N]} proved that
$H_n$ is an integer only for $n=1$. Erd\H{o}s and Niven
\cite{[EN]} proved that $H(n,k)$ is an integer only for finitely
many positive integers $n$ and $k$, and later Chen and
Tang \cite{[CT]} showed that $H(1,1)$ and $H(3,2)$ are
the only integral values. See also \cite{[FHJY],[HW],[LHQW],[WH], [YLFJ]}
for some further studies on this topic. In 1862, Wolstenholme
\cite{[Wol]} proved that for any prime $p\ge 5$ the numerator of
$H_{p-1}$ is divisible by $p^2$. Eswarathasan and Levine \cite{[Es]}
conjectured that the set of positive integers $n$ such that the numerator of
$H_n$ is divisible by $p$ is finite for any given prime $p$. Boyd \cite{[Bo]}
confirmed this conjecture for all primes $p\le 547$ but 83, 127 and 397.
He also conjectured that $v_p(H_n)\le 3$ holds for any integer $n\ge 1$
and for any odd prime $p\ge 5$.
See \cite{[Ka],[Sa]} for more results on the $p$-adic
properties of $H_n$.

Let $p$ be a prime and $n$ be a positive integer. Lengyel \cite{[L]} showed that
there exists a constant $c'=c'(k,p)> 0$ so that for any $n\ge n_0(k,p)$,
one has $v_p(s(n,k))\ge c'n$. This implies that the $p$-adic valuation of
$H(n,k)$ has a lower bound. Consequently, Leonetti and Sanna \cite{[LS]}
conjectured that there exists a positive constant $c=c(k,p)$ such that
\begin{align}
v_p(H(n,k))<-c\log n\label{1.1}
\end{align}
for all large $n$ and confirmed this conjecture for some special cases.
Moreover, Komatsu and Young \cite{[KY]} used the theory of $p$-adic
Newton polygon to show that if $k$ is a nonnegative integer and $n$ is
of the form $n=kp^r+m$ with $0\le m<p^r$, then $v_p(s(n+1,k+1))=v_p(n!)-v_p(k!)-kr$.
Using the study of the higher order Bernoulli numbers $B_n^{(l)}$,
Adelberg \cite{[Ad]} investigated some $p$-adic properties of Stirling
numbers of both kinds. Qiu and Hong \cite{[QH]} presented a detailed 2-adic
analysis and obtained an exact formula for $v_2(s(2^n, k))$ with $n$ and
$k$ being positive integers such that $1\le k\le 2^n$. In \cite{[QFH]},
Qiu, Feng and Hong provide a 3-adic analysis and arrive at
a formula for $v_3(s(a3^n, k))$ with $k$ being an integer
such that $1\le k\le a3^n$, where $a\in\{1, 2\}$.

In this paper, we are mainly concerned with the $p$-adic
valuations of the Stirling numbers of the first kind. Actually,
we yield an exact expression or a lower bound of
$v_p(s(ap, i))$ with $p\ge 5$ being a prime and $a$ and $i$
being integers such that $1\le a\le p-1$ and $1\le i\le ap$.
Throughout let $p\ge 5$ be a prime and $n$ be a positive
integer. Let $a$ be an integer with $1\le a\le p-1$.
Obviously, we have $v_p(s(ap,1))=v_p((ap-1)!)=a-1$ and
$$
v_p(s(ap,ap-1))=v_p\big(\binom{ap}{2}\big)=1,\ v_p(s(ap,ap))=v_p(1)=0.
$$
For any integer $i$ with $2\le i\le ap-2$, one can always write
$i=ap-k$ for some integer $k$ with $2\le k\le ap-2$.

As usual, for any given real number $y$, $\lfloor y\rfloor$ and $\lceil y\rceil$
stand for the largest integer no more than $y$ and the smallest
integer no less than $y$, respectively. For any integer $k$
and prime $p$, let $\langle k\rangle$ denote the
integer such that $0\le \langle k\rangle\le p-2$
and $k\equiv \langle k\rangle \pmod {p-1}$.
Let $\epsilon_k :=0$ if $k$ is even
and $\epsilon_k :=1$ if $k$ is odd.
The $n$-th Bernoulli number $B_n$ is defined
by the Maclaurin series as
$$
\frac{x}{e^x-1}=\sum_{n=0}^{\infty}B_n\frac{x^n}{n!}.
$$
By the von Staudt-Clausen theorem (see, for instance,
\cite{[IR]}), we know that if $n$ is even, then
$$
B_n+\sum_{(p-1)\mid n}\frac{1}{p}\in \mathbb{Z},
$$
where the sum is over all primes $p$ such that
$(p-1)\mid n$. The von Staudt-Clausen theorem tells us
that $v_p(B_{l})\ge 0$ for all even integers $l$
with $2\le l\le p-3$. In particular, if $v_p(B_{l})=0$
for all even integers $l$ with $2\le l\le p-3$, then
$p$ is called a {\it regular prime}. Note that 3 is
a regular prime. If $p$ is not regular it is called
{\it irregular}. The first irregular primes are 37
and 59. We should emphasize that the cardinality
and density of regular primes are largely unknown.
Although one can argue heuristically that asymptotically
more than half of primes should be regular, it is not even
known that there are infinitely many regular primes,
while infinitely many primes are known to be irregular.
We refer the interested readers to \cite{[IR]} and
\cite{[Was1]} for the history and basic facts on the
regular and irregular primes.
We can now state the first main result of this paper as follows.
\begin{thm}\label{thm1}
Let $p\ge 5$ be a prime and let $a$ and $k$ be integers
such that $1\le a\leq p-1$ and $2\le k \leq ap-2$.
Then each of the following is true:

{\rm (i).} If $k\equiv \epsilon_k\pmod{p-1}$, then
\begin{align*}
 v_p(s(ap, ap-k)) = (v_p(k)+1)\epsilon_k.
\end{align*}

{\rm (ii).} If $2\le k \leq a(p-1)-1$ and $k\not\equiv\epsilon_k\pmod{p-1}$,
then
\begin{align*}
v_p(s(ap, ap-k)) \ge (v_p(k)+1)\epsilon_k+1,
\end{align*}
where the equality holds if and only if
$v_p\big(B_{2\big\lfloor\frac{\langle k\rangle}{2}\big\rfloor}\big)=0$.
In particular, if $p$ is regular, then
$v_p(s(ap,ap-k)) = (v_p(k)+1)\epsilon_k+1.$

{\rm (iii).} If $a\ge 4$ and $a(p-1)+2\le k\leq ap-2$, then
\begin{align*}
v_p(s(ap, ap-k)) \geq a+k-ap.
\end{align*}
\end{thm}

Let $a$ and $k$ be integers such that $1\le a\leq p-1$
and $\max\{1,a-1\}\le k\le ap$. From Theorem \ref{thm1},
we can see that if $p$ is regular, then $v_p(s(ap,k))\le 2+v_p(k)\le 3$.
Since $s(ap,k)=(ap-1)!H(ap-1,k-1)$, the following consequence follows immediately.

\begin{cor}\label{cor1.3}
Let $p\ge 5$ be a regular prime and let $a$ and $k$ be integers such
that $1\le a\leq p-1$ and $\max\{1,a-2\}\le k \leq ap-1$.
Then $v_p(H(ap-1,k)) \le 4-a.$
\end{cor}

\noindent
So for any regular prime $p\ge 7$ and for arbitrary integers $a$
and $k$ with $5\le a\le p-1$ and $a-2\le k\le ap-1$, one has
$v_p(H(ap-1,k)) < -c\log{(ap-1)}$ with $c=\frac{1}{2\log p}$.
This gives a partial support to Conjecture (\ref{1.1})
due to Leonetti and Sanna.

Consequently, we state the second main result of this paper as follows.

\begin{thm}\label{thm2}
Let $p\ge 5$ be a prime and let $a$ and $n$ be positive integers such that
$(a,p)=1$. Let $k$ be an odd integer with $1\le k\le ap^n-1$. Then
\begin{align*} 
&v_p(s(ap^n,ap^n-k))\notag\\
&\left\{\begin{array}{ll}
=v_p(s(ap^n,ap^n-k+1))+v_p(ap^n-k)+n, & {\it if}\ v_p(s(ap^n,ap^n-k+1))\le 2n-1; \\
\ge v_p(ap^n-k)+3n, & {\it if}\  v_p(s(ap^n,ap^n-k+1))\ge 2n.
\end{array}\right.
\end{align*}
\end{thm}

On the other hand, Lengyel \cite{[L]} proved the following interesting result.

\begin{thm}\label{thm1.4'} \cite{[L]}
For any prime $p$, any integer $a\ge 1$ with $(a,p)=1$,
and any even $k\ge 2$ with the condition
\begin{align*}
\exists \ n_1\in {\mathbb Z}^+: n_1> 3\log_p{k}+\log_p{a}\ \ \
{\it such\ that}\  \ \ v_p(s(ap^{n_1},ap^{n_1}-k))< n_1
\end{align*}
or $k=1$ with $n_1=1$, then for $n\ge n_1$ one has
\begin{align*}
v_p(s(ap^{n+1},ap^{n+1}-k))=v_p(s(ap^n,ap^n-k))+1.
\end{align*}
\end{thm}
\noindent
Meanwhile, for any odd
$k\ge 3$, Lengyel conjectured in \cite{[L]} that for any
integer $n\ge n_1(p,k)$ with some sufficiently large $n_1(p,k)$,
one has
\begin{align}\label{1.7}
v_p(s(ap^{n+1},ap^{n+1}-k))=v_p(s(ap^n,ap^n-k))+2.
\end{align}

Now by Theorems \ref{thm2} and \ref{thm1.4'},
one yields the following analogous result which is the third
main result and proves partially Lengyel's Conjecture (\ref{1.7}).

\begin{thm}\label{thm1.5}
Let $p\ge 5$ be a prime and let $a$ be a positive integer such that
$(a,p)=1$. Let $k\ge 3$ be an odd integer with the condition
\begin{align*}
\exists \ n_1\in {\mathbb Z}^+: n_1> 3\log_p{(k-1)}+\log_p{a}\ \ \
{\it such\ that} \  \ \ v_p(s(ap^{n_1},ap^{n_1}-(k-1)))< n_1.
\end{align*}
Then for any positive integer $n$ with $n\ge n_1$ one has
$$v_p(s(ap^{n+1},ap^{n+1}-k))=v_p(s(ap^n,ap^n-k))+2.$$
Furthermore, we have
$$v_p(s(ap^{n},ap^{n}-k))=v_p(s(ap^{n_1},ap^{n_1}-k))+2(n-n_1).$$
\end{thm}

The paper is organized as follows. We reveal some useful
properties of Stirling numbers of the first kind and
generalized harmonic numbers in the next section.
Then we prove Theorems  \ref{thm2} and \ref{thm1.5}
and Theorem \ref{thm1} in Sections 3 and 4, respectively.
Finally, for any positive integer $n$ and $k$ such that
$2\le k\le ap^n-2$, we propose three conjectures on
the $p$-adic valuation of $s(ap^n,k)$.

\section{Preparatory lemmas on Stirling numbers of the first kind
and generalized harmonic numbers}

Let $n$ and $k$ be positive integers. Some basic identities involving
Stirling numbers of the first kind can be listed as follows (see \cite{[LC]}):
$$s(n+1,k)=ns(n,k)+s(n,k-1),\ s(n,0)=0,\ s(n,k)=0 \ {\rm if}\ k\ge n+1,$$
$$
\sum_{k=1}^ns(n,k)=n!,\ s(n,1)=(n-1)!,\ s(n,n)=1,\ s(n,n-1)=\binom{n}{2},
$$
$$s(n,2)=(n-1)!\sum_{k=1}^{n-1}\frac{1}{k} \
{\rm and}\ s(n,n-2)=\frac{1}{4}(3n-1)\binom{n}{3}\ {\rm if}\ n\ge 2.$$

The following two results are known.

\begin{lem}\label{lem1} \cite{[VA]}
Let $n$ and $k$ be positive integers. If $n+k$ is odd, then
$$ s(n,k)=\frac{1}{2}\sum_{i=k+1}^{n}(-1)^{n-i}n^{i-k}\binom{i-1}{i-k}s(n,i).$$
\end{lem}

\begin{lem}\label{lem6'} \cite{[Bo]}
Let $p$ be a prime with $p\ge 5$. Let $n$ and $k$ be
positive integers with $k\le p-1$. Then
\begin{align*}
H_{np}\equiv \frac{1}{p}H_n \pmod{p^2}
\end{align*}
and
\begin{align*}
H_{np+k}\equiv H_{np}+H_k \pmod{p}.
\end{align*}
\end{lem}

As introduced in \cite{[QH]}, for any given nonnegative integer $m$,
we define the {\it $m$-th Stirling numbers of the first kind}, denoted by
$s_m(n, k)$ with $n$ and $k$ being nonnegative integers, as follows:
\begin{align*}
(x+m)(x+m+1)\cdots(x+m+n-1):= \sum_{k=0}^{n} s_m(n,k)x^k.
\end{align*}
One notices that the $m$-th Stirling number of the first kind
is actually a special case of the generalized Stirling number
of Hsu and Shiue \cite{[HS]}. The following two basic results
on $s_m(n,k)$ will be needed in the proof of Theorem \ref{thm1}.

\begin{lem}\label{lem2}\cite{[QH]}
Let $m, k$ and $n$ be nonnegative integers. Then
$$  s(m+n,k)=\sum_{i=0}^{k}s(m,i) s_m(n,k-i). $$
\end{lem}

\begin{lem}\label{lem3}\cite{[QH]}
Let $m, k$ and $n$ be nonnegative integers. Then
$$  s_m(n,k)=\sum_{i=k}^{n}s(n,i)\binom{i}{i-k}m^{i-k}.$$
\end{lem}

Let $n$ be a positive integer and let $r$ be a nonnegative integer.
Then the  {\it generalized harmonic number}, denoted by $H_{n}^{(r)}$,
is defined by
$$
H_{n}^{(r)}:=\sum_{i=1}^{n}\frac{1}{i^r}.
$$
Washington proved the following congruences
which were later extended by Hong \cite{[H]}.

\begin{lem}\label{lem4} \cite{[Was]}
Let $p$ be an odd prime and let $r$ be an integer.

{\rm (i).} If $r\ge1$ is odd and $p\geq r+4$, then
$$ H_{p-1}^{(r)}\equiv -\frac{r(r+1)}{2(r+2)}B_{p-r-2}p^2 \pmod {p^3}.$$

{\rm (ii).} If $r\geq2$ is even and $p\geq r+3$, then
$$ H_{p-1}^{(r)}\equiv \frac{r}{r+1}B_{p-r-1}p \pmod {p^2}. $$
\end{lem}

\noindent
From Lemma \ref{lem4} and noticing the fact that
$v_p\big(B_{l}\big)\ge 0$ when $(p-1)\nmid l$,
one can derive the following corollary immediately.

\begin{cor}\label{cor}
Let $p\ge 5$ be a prime. Let $r$ be an integer with $1\le r\le p-3$.
Then
\begin{align}
v_p\big(H_{p-1}^{(r)}\big)\ge 1+\epsilon_{r} \label{2.1}
\end{align}
with the equality holding if and only if
$v_p\big(B_{p-1-2\lceil\frac{r}{2}\rceil}\big)=0$.
\end{cor}

Let $n$ and $k$ be integers with $n\ge 1$ and $k\ge 0$. Let $S_k=S_k(x_1, ..., x_n)$
and $\sigma_k=\sigma_k(x_1, ..., x_n)$ denote the $k$th {\it Newton sum} and the $k$th
{\it elementary symmetric function} of the $n$ variables $x_1, ..., x_n$, respectively.
That is,
$$
S_k=S_k(x_1, ..., x_n):=\sum_{i=1}^n x_i^k
$$
and
$$
\sigma_k=\sigma_k(x_1, ..., x_n):=\left\{\begin{array}{cl}1, & {\rm if} \ k=0,\\
\sum_{1\le i_1<\cdots<i_k\le n}x_{i_1}\cdots x_{i_k}, & {\rm if} \ 1\le k\le n,\\
0, & {\rm if} \ k>n.
\end{array}\right.
$$
Then the well known Newton-Girard formula can be stated as follows.

\begin{lem}\label{lem4'}\cite{[Ser]}
Let $n$ be a positive integer and let $k$ be a nonnegative integer.

{\rm (i).} If $1\le k\le n$, then
$$
\sum_{i=1}^k(-1)^{k-i}S_i\sigma_{k-i}+(-1)^kk\sigma_k=0.
$$

{\rm (ii).} If $k > n$, then
$$
\sum_{i=0}^n(-1)^{i}S_{k-i}\sigma_{i}=0.
$$
\end{lem}

\noindent
For any integer $r$ with $0\le r\le p-1$, it is easy to see that
$$
S_r\Big(1,\frac{1}{2},\cdots,\frac{1}{p-1}\Big)=\sum_{i=1}^{p-1}\frac{1}{i^r}=H_{p-1}^{(r)}
$$
and
$$
\sigma_r\Big(1,\frac{1}{2},\cdots,\frac{1}{p-1}\Big)=H(p-1,r),$$
where $H(p-1, 0)=1$. Therefore by using Lemma \ref{lem4'} and
Corollary \ref{cor}, we can deduce the following result.

\begin{lem}\label{lem5'}
Let $p$ be an odd prime and let $r$ be an integer.

{\rm (i).} If $r\ge1$ is odd and $p\geq r+4$, then
\begin{align}
H(p-1,r)\equiv \frac{1}{r}H_{p-1}^{(r)} \pmod {p^3}.\label{2.2}
\end{align}

{\rm (ii).} If $r\geq2$ is even and $p\geq r+3$, then
\begin{align}
H(p-1,r)\equiv -\frac{1}{r}H_{p-1}^{(r)} \pmod {p^2}. \label{2.3}
\end{align}

\end{lem}

\begin{proof}
Since $H(p-1,1)=H_{p-1}^{(1)}$, Lemma \ref{lem5'} is obviously true when $r=1$.
Now let $2\le r\le p-3$. Using Lemma \ref{lem4'} and also noting that $H(p-1,0)=1$,
one derives that
\begin{align}
H(p-1,r)&=\frac{1}{r}\sum_{i=1}^{r}(-1)^{i-1}H(p-1,r-i)H_{p-1}^{(i)}\notag\\
&=\frac{1}{r}\sum_{i=1}^{r-1}(-1)^{i-1}H(p-1,r-i)H_{p-1}^{(i)}-\frac{(-1)^{r}}{r}H_{p-1}^{(r)}. \label{2.4}
\end{align}

If $r=2$, then it follows from (\ref{2.4}) that
\begin{align}
H(p-1,r)=H(p-1,2)
=\frac{1}{2}H(p-1,1)H_{p-1}^{(1)}-\frac{1}{2}H_{p-1}^{(2)}
=\frac{1}{2}\big(H_{p-1}^{(1)}\big)^2-\frac{1}{2}H_{p-1}^{(2)}.\label{2.5}
\end{align}
From (\ref{2.1}), one yields that $v_p\big(H_{p-1}^{(2)}\big)\ge1$ and
\begin{align}
v_p\left(\big(H_{p-1}^{(1)}\big)^2\right)=2v_p\big(H_{p-1}^{(1)}\big)\ge 4.\label{2.6}
\end{align}
Using (\ref{2.5}) and (\ref{2.6}), we then deduce that
\begin{align*}
H(p-1,r)=H(p-1,2)\equiv -\frac{1}{2}H_{p-1}^{(2)}=-\frac{1}{r}H_{p-1}^{(r)}\pmod {p^4}.
\end{align*}
Hence Lemma \ref{lem5} is true when $r=2$.

Now let $3\le r\le p-3$. Suppose that Lemma \ref{lem5'} holds for any integer $e$
with $2\le e\le r-1$. In the following, we show that Lemma \ref{lem5'} is still
true for the $r$ case.

For any integer $i$ with $1\le i \le r-1$,
by the inductive hypothesis and Corollary \ref{cor}, one obtains that
$v_p(H(p-1,r-i))\ge 1$ and $v_p\big(H_{p-1}^{(i)}\big)\ge 1$.

If $r$ is even, then one can conclude that
$v_p\big(H(p-1,r-i)H_{p-1}^{(i)}\big)\ge 2$. Since $2\le r\le p-3$, one
has $\frac{1}{r}$ is a $p$-adic unit and so $v_p(\frac{1}{r})=0$. Hence
\begin{align}
v_p\Big(\frac{1}{r}\sum_{i=1}^{r-1}(-1)^{i-1}H(p-1,r-i)H_{p-1}^{(i)}\Big)\ge 2.\label{2.7}
\end{align}
Thus (\ref{2.4}) and (\ref{2.7}) imply the truth of (\ref{2.3}).
So (\ref{2.3}) is proved.

If $r$ is odd, then one of $i$ and $r-i$ must be odd. So it
follows from Corollary \ref{cor} and the inductive
hypothesis that either $v_p\big(H_{p-1}^{(i)}\big)\ge 2$
or $v_p(H(p-1,r-i))\ge 2$. Hence
$v_p\big(H(p-1,r-i)H_{p-1}^{(i)}\big)\ge 3$,
which infers that
\begin{align}
v_p\Big(\frac{1}{r}\sum_{i=1}^{r-1}(-1)^{i-1}H(p-1,r-i)H_{p-1}^{(i)}\Big)\ge 3.\label{2.8}
\end{align}
Therefore (\ref{2.2}) follows from (\ref{2.4}) and (\ref{2.8}).

This completes the proof of Lemma \ref{lem5'}.
\end{proof}

Combining Lemma \ref{lem4} with Lemma \ref{lem5'},
one can easily obtain the following congruences.

\begin{lem}\label{lem5}
Let $p$ be an odd prime and let $r$ be an integer.

{\rm (i).} If $r\ge1$ is odd and $p\geq r+4$, then
$$ H(p-1,r)\equiv -\frac{r+1}{2(r+2)}B_{p-r-2}p^2 \pmod {p^3}.$$

{\rm (ii).} If $r\geq2$ is even and $p\geq r+3$, then
$$ H(p-1,r)\equiv -\frac{1}{r+1}B_{p-r-1}p \pmod {p^2}. $$
\end{lem}

Evidently, one has $s(p,1)=(p-1)!\equiv -1\pmod{p}$
by Wilson theorem, and $s(p,p)=1$. Also it is known
that $s(p,k)\equiv 0 \pmod{p}$ if $2\le k\le p-1$ (see, for example, \cite{[LC]}).
Now by using Lemma \ref{lem5}, we can derive the following
lower bound on the $p$-adic valuation of $s(p, k)$ for
any integer $k$ with $2\le k\le p-1$.

\begin{cor}\label{cor2.10}
Let $p\ge 5$ be a prime and let $k$ be an integer with $1\le k\le p$.
Then
\begin{align}
v_p(s(p,k))=
\left\{\begin{array}{cl}0, & {\it if} \ k=1\ {\it or}\ k=p,
\\1, & {\it if} \ k=p-1,
\end{array}\right.\label{2.9}
\end{align}
and if $2\le k\le p-2$, then
\begin{align}
v_p(s(p,k))\ge 1+\epsilon_{k-1} \label{2.10}
\end{align}
with the equality holding if and only if
$v_p\big(B_{p-1-2\lfloor\frac{k}{2}\rfloor}\big)=0$.
\end{cor}

\begin{proof}
Since $v_p(s(p,1))=v_p((p-1)!)=0$,
$v_p(s(p,p-1))=v_p\left(\binom{p}{2}\right)=1$ and
$v_p(s(p,p))=v_p(1)=0$, the equality (\ref{2.9}) is clearly true.

Now let $2\le k\le p-2$. Note that $v_p\big(B_{p-1-2\lfloor\frac{k}{2}\rfloor}\big)\ge 0$.
So replacing $r$ by $k-1$ in Lemma \ref{lem5} gives us that
\begin{align}
v_p(H(p-1,k-1))\ge1+\epsilon_{k-1}\label{2.11}
\end{align}
with the equality holding if and only if
$v_p\big(B_{p-1-2\lfloor\frac{k}{2}\rfloor}\big)=0$.
Since
$$
v_p(s(p,k))=v_p((p-1)!H(p-1,k-1))=v_p(H(p-1,k-1)),
$$
(\ref{2.10}) follows immediately from  (\ref{2.11}).

This finishes the proof of Corollary \ref{cor2.10}.
\end{proof}

For any positive integer $m$, it follows from Lemma \ref{lem3}
that $s_m(n,k)\equiv s(n,k)\pmod {m}$. Now together
with Theorem \ref{thm2}, we can obtain the following result.

\begin{lem}\label{lem7}
Let $p\ge 5$ be a prime. Let $a$ and $n$ be positive integers such that
$(a,p)=1$. For any integer $k$ with $0\le k\le ap^n$,
if $2\nmid k$, then
\begin{align}\label{2.12}
v_p(s_{p^n}(ap^n, ap^n-k))\ge n,
\end{align}
and if $2\mid k$, then
\begin{align}\label{2.13}
s_{p^n}(ap^n, ap^n-k)\equiv s(ap^n, ap^n-k) \pmod{p^{2n}}.
\end{align}
\end{lem}

\begin{proof}
Since $s_{p^n}(ap^n, ap^n)=s(ap^n, ap^n)=1$, (\ref{2.13}) is
obviously true when $k=0$. If $k=1$, then by Lemma
\ref{lem3}, one obtains that
\begin{align*}
s_{p^n}(ap^n, ap^n-1)=s(ap^n, ap^n-1)+ap^{2n}.
\end{align*}
Since $v_p(s(ap^n, ap^n-1))=n$, it follows that
$v_p(s_{p^n}(ap^n, ap^n-1))=n$. Thus (\ref{2.12}) is proved when $k=1$.
For the case that $k=ap^n$, we have
$$
v_p(s_{p^n}(ap^n, 0))=
v_p\Big(\frac{((a+1)p^n-1)!}{(p^n-1)!}\Big)
=\frac{ap^n-d_p(a)}{p-1}\ge \frac{a(p^n-1)}{p-1}
=a\sum_{i=0}^{n-1}p^i\ge an.
$$
Hence (\ref{2.12}) is clearly true when $k=ap^n$ is odd.
If $k=ap^n$ is even, then $a$ is even and so $a\ge 2$,
which infers that $v_p(s_{p^n}(ap^n, 0))\ge 2n$.
Thus together with the fact that $s(ap^n,0)=0$,
we can derive that (\ref{2.13}) holds when $k$ is even.

Now let $2\le k\le ap^n-1$. From Lemma \ref{lem3}, one deduces that
\begin{align}\label{2.14}
s_{p^n}(ap^n, ap^n-k)=
&s(ap^n, ap^n-k)+s(ap^n, ap^n-k+1)(ap^n-k+1)p^n\notag\\
&+\sum_{i=ap^n-k+2}^{ap^n}s(ap^n, i)\binom{i}{i-ap^n+k}p^{n(i-ap^n+k)}\notag\\
\equiv &s(ap^n, ap^n-k)+s(ap^n, ap^n-k+1)(ap^n-k+1)p^n \pmod {p^{2n}}.
\end{align}
If $k$ is odd, then by Theorem \ref{thm2}, one gets that $v_p(s(ap^n, ap^n-k))\ge n$.
Therefore it follows from (\ref{2.14}) that $v_p(s_{p^n}(ap^n, ap^n-k))\ge n$
as desired. If $k$ is even, then $k-1$ is odd and so $v_p(s(ap^n, ap^n-k+1))
=v_p(s(ap^n, ap^n-(k-1)))\ge n$. Thus (\ref{2.14}) implies the
truth of (\ref{2.13}) as required.

The proof of Lemma \ref{lem7} is complete.
\end{proof}

\section{Proofs of Theorems \ref{thm2} and \ref{thm1.5}}

In this section, we present the proofs of Theorems \ref{thm2}
and \ref{thm1.5}. We begin with the proof of Theorem \ref{thm2}. \\

\noindent
\textbf{Proof of Theorem \ref{thm2}.}
Let $p\ge 5$ be a prime and let $a$ and $n$ be
positive integers such that $(a,p)=1$. Let $k$ be an odd
integer with $1\le k\le ap^n-1$. Since $k$ is odd, replacing
in Lemma \ref{lem1} $n$ and $k$ by $ap^n$ and $ap^n-k$,
respectively, gives us that
\begin{align}\label{3.1}
s(ap^n, ap^n-k)&=\frac{1}{2}\sum_{i=ap^n-k+1}^{ap^n}(-1)^{ap^n-i}(ap^n)^{i-ap^n+k}
\binom{i-1}{i-ap^n+k}s(ap^n, i).
\end{align}
It follows immediately from (\ref{3.1}) that $v_p(s(ap^n, ap^n-k))\ge n$.

Note that $p\ge 5$ and $(a,p)=1$. If $k=1$, then
$$
v_p(s(ap^n, ap^n-k+1))=v_p(s(ap^n, ap^n))=v_p(1)=0
$$
and
$$
v_p(s(ap^n, ap^n-k))=v_p(s(ap^n, ap^n-1))
=v_p\Big(\binom{ap^n}{2}\Big)
=v_p\Big(\frac{ap^n(ap^n-1)}{2}\Big)=n.
$$
So Theorem \ref{thm2} is clearly true when $k=1$.

Now let $3\le k\le ap^n-1$. By (\ref{3.1}), one deduces that
\begin{align}\label{3.2}
s(ap^n, ap^n-k)= \frac{1}{2}(ap^n(ap^n-k)s(ap^n, ap^n-k+1)+L),
\end{align}
where
\begin{align} \label{3.3}
L=\sum_{i=ap^n-k+2}^{ap^n-k+3}(-1)^{ap^n-i}(ap^n)^{i-ap^n+k}
\binom{i-1}{i-ap^n+k}s(ap^n, i)+\Delta
\end{align}
with $\Delta$ being an integer such that $v_p(\Delta)\ge 4n$.
Since $p\ge 5$ and $(a,p)=1$ and $3\le k\le ap^n-1$,
it is easy to check that for $i=ap^n-k+2$ and $i=ap^n-k+3$, one has
\begin{align}\label{3.4}
v_p\Big(\binom{i-1}{i-ap^n+k}\Big)\ge v_p(ap^n-k).
\end{align}
Notice that $k-2$ is odd. Then
\begin{align}\label{3.5}
v_p(s(ap^n, ap^n-k+2))=v_p(s(ap^n, ap^n-(k-2)))\ge n.
\end{align}
We claim that
\begin{align}\label{3.6}
v_p(\Delta)\ge v_p(ap^n-k)+3n.
\end{align}
Then it follows from (\ref{3.3}) to (\ref{3.5}) and claim (\ref{3.6}) that
\begin{align}\label{3.7}
v_p(L)&\ge \min\Big\{v_p\Big(\sum_{i=ap^n-k+2}^{ap^n-k+3}
(-1)^{ap^n-i}(ap^n)^{i-ap^n+k}\binom{i-1}{i-ap^n+k}s(ap^n, i)\Big),
v_p(\Delta)\Big\}\notag\\
&\ge v_p(ap^n-k)+3n.
\end{align}

If $v_p(s(ap^n, ap^n-k+1))\ge 2n$, then one can easily get that
\begin{align}\label{3.8}
v_p(ap^n(ap^n-k)s(ap^n, ap^n-k+1))\ge v_p(ap^n-k)+3n.
\end{align}
Hence by (\ref{3.2}), (\ref{3.7}) and (\ref{3.8}),
we obtain that
\begin{align*}
v_p(s(ap^n, ap^n-k))
&\ge \min\{v_p(ap^n(ap^n-k)s(ap^n, ap^n-k+1)),v_p(L)\}\\
&\ge v_p(ap^n-k)+3n
\end{align*}
as desired.

If $v_p(s(ap^n, ap^n-k+1))\le 2n-1$, then from (\ref{3.7}) one derives that
\begin{align}\label{3.9}
v_p(ap^n(ap^n-k)s(ap^n, ap^n-k+1))\le v_p(ap^n-k)+3n-1< v_p(L).
\end{align}
Since $p\ge 5$, by (\ref{3.2}) and (\ref{3.9}) together
with the isosceles triangle principle (see, for example, \cite{[K]}),
we arrive at
\begin{align*}
v_p(s(ap^n, ap^n-k))&=v_p(ap^n(ap^n-k)s(ap^n, ap^n-k+1)+L)\\
&=v_p(ap^n(ap^n-k)s(ap^n, ap^n-k+1))\\
&=v_p(s(ap^n, ap^n-k+1))+v_p(ap^n-k)+n
\end{align*}
as expected. So to finish the proof of Theorem \ref{thm2},
it remains to show the truth of claim (\ref{3.6}) that
will be done in what follows.

If $k=3$, then $\Delta$ is empty sum and so
$\Delta=0$. Since $v_p(0)=+\infty$, claim (\ref{3.6}) is true
if $k=3$. In the following one lets $k\ge 5$. One has
\begin{align}\label{3.10'}
\Delta=\sum_{i=ap^n-k+4}^{ap^n}(-1)^{ap^n-i}\Delta_i
\end{align}
where
\begin{align*}
\Delta_i:=(ap^n)^{i-ap^n+k}\binom{i-1}{i-ap^n+k}s(ap^n, i).
\end{align*}
Let $i$ be any integer with $ap^n-k+4\le i\le ap^n$. Then one
can write $i:=ap^n-k+j$ for some integer $j$ with $4\le j\le k$. So
\begin{align}\label{3.10}
v_p(\Delta_i)\ge nj+v_p\Big(\binom{ap^n-k+j-1}{j}\Big).
\end{align}
In what follows, we show that the following $p$-adic estimate holds:
\begin{align}\label{3.11}
v_p(\Delta_i)\ge v_p(ap^n-k)+3n.
\end{align}
The proof of (\ref{3.11}) is divided into the following two cases.

{\bf Case 1.} $v_p(ap^n-k)\le n$. It follows from the
hypothesis $j\ge 4$ and (\ref{3.10}) that
$v_p(\Delta_i)\ge jn\ge 4n\ge v_p(ap^n-k)+3n$.
Thus (\ref{3.11}) is true in this case.

{\bf Case 2.} $v_p(ap^n-k)\ge n+1$. We have
\begin{align*}
\binom{ap^n-k+j-1}{j}=\frac{(ap^n-k+j-1)(ap^n-k+j-2)\cdots(ap^n-k)}{j!}.
\end{align*}

{\bf Subcase 2.1.} $v_p(j')< v_p(ap^n-k)$ for any integer $j'$
with $1\le j'\le j-1$. Then $v_p(ap^n-k+j')=v_p(j')$. It implies that
\begin{align}\label{3.12}
v_p\Big(\binom{ap^n-k+j-1}{j}\Big)=v_p(ap^n-k)-v_p(j).
\end{align}
Since $j\ge 4$ and $p\ge 5$, one can deduce that $j\ge v_p(j)+4$.
In fact, if $v_p(j)=0$, then $j\ge v_p(j)+4$ and if $v_p(j)\ge 1$,
then $j\ge p^{v_p(j)}\ge v_p(j)+4$ as expected. Now by
(\ref{3.10}) and (\ref{3.12}), we have
\begin{align*}
v_p(\Delta_i)
&\ge nj+v_p(ap^n-k)-v_p(j)\\
&\ge v_p(ap^n-k)+4n+v_p(j)(n-1)\\
&>v_p(ap^n-k)+3n
\end{align*}
as required. So (\ref{3.11}) holds in this case.

{\bf Subcase 2.2.} $v_p(j')\ge v_p(ap^n-k)\ge n+1$ for some integer
$j'$ with $1\le j'\le j-1$. Then $p^{v_p(ap^n-k)}\ge v_p(ap^n-k)+2$
and by (\ref{3.10}), one has
\begin{align*}
v_p(\Delta_i)&\ge nj\ge nj'+n\ge np^{v_p(j')}+n
\ge np^{v_p(ap^n-k)}+n\ge v_p(ap^n-k)+3n
\end{align*}
as desired. Therefore (\ref{3.11}) is true in this case.
So (\ref{3.11}) is proved.

Finally, from (\ref{3.10'}) and (\ref{3.11}) one can deduce immediately that
$$v_p(\Delta)\ge \min_{ap^n-k+4\le i\le ap^n}\{v_p(\Delta_i)\}\ge v_p(ap^n-k)+3n$$
as (\ref{3.6}) claimed. This completes the proof of claim
(3.6) and that of Theorem \ref{thm2}.  \qed \\

Let $i$ be an integer such that $1\le i\le ap^n-1$. We remark
that if $ap^n+i$ is odd, then using Theorem \ref{thm2}
we can deduce that $v_p(s(ap^n,i))\ge n$ since
one may write $i=ap^n-i'$ for some integer $i'$
with $1\le i'\le ap^n-1$, where $i'$ is odd
if and only if $ap^n+i$ is odd.

We can now use Theorem \ref{thm2} to show Theorem \ref{thm1.5}. \\

\noindent
\textbf{Proof of Theorem \ref{thm1.5}.}
Let $p$ be a prime with $p\ge 5$. Let $a$ and $k$ be positive
integers such that $(a,p)=1$ and $k\ge 3$ being odd with
the condition
\begin{align}\label{3.13}
\exists n_1\in {\mathbb Z}^+: n_1> 3\log_p{(k-1)}+\log_p{a}\ \ \ {\rm such\ that}
\  \ \ v_p(s(ap^{n_1},ap^{n_1}-(k-1)))< n_1.
\end{align}
From condition (\ref{3.13}), one can easily get that $k<p^{n_1}$,
which infers that $v_p(k)<n_1$.

Let $n$ be an integer with $n\ge n_1$. Since $k-1\ge 2$ is even and
condition (\ref{3.13}) holds for $k-1$, it follows from Theorem \ref{thm1.4'} that
\begin{align*}
v_p(s(ap^{n}, ap^{n}-(k-1)))
=v_p(s(ap^{n_1}, ap^{n_1}-(k-1)))+n-n_1
< n
\end{align*}
and
\begin{align}\label{3.14}
v_p(s(ap^{n+1}, ap^{n+1}-(k-1)))
&=v_p(s(ap^n, ap^n-(k-1)))+1\\
&=v_p(s(ap^{n_1}, ap^{n_1}-(k-1)))+n+1-n_1\notag\\
&< n+1\notag.
\end{align}
Note that $k$ is odd and $v_p(k) < n_1\le n$. Since
$v_p(s(ap^{n+1}, ap^{n+1}-(k-1))) < n+1$ and
$v_p(s(ap^{n}, ap^{n}-(k-1)))<n$, Theorem \ref{thm2}
together with (\ref{3.14}) give us that
\begin{align*}
v_p(s(ap^{n+1}, ap^{n+1}-k))
&=v_p(s(ap^{n+1}, ap^{n+1}-k+1))+v_p(ap^{n+1}-k)+n+1\\
&=v_p(s(ap^{n}, ap^{n}-k+1))+v_p(k)+n+2\\
&=v_p(s(ap^{n}, ap^{n}-k+1))+v_p(ap^n-k)+n+2\\
&=v_p(s(ap^{n}, ap^{n}-k))+2
\end{align*}
as Theorem \ref{thm1.5} expected. Hence
\begin{align*}
v_p(s(ap^{n}, ap^{n}-k))=v_p(s(ap^{n_1}, ap^{n_1}))+2(n-n_1).
\end{align*}

This concludes the proof of Theorem \ref{thm1.5}.  \qed

\section{Proof of Theorem \ref{thm1}}

In this section, we use the lemmas given in section 2 and Theorem \ref{thm2}
to supply the proof of Theorem \ref{thm1}.\\

\noindent
\textbf{Proof of Theorem \ref{thm1}.}
We prove Theorem \ref{thm1} by induction on the positive integer $a$ with $a\le p-1$.

First of all, let $a=1$. Then $2\le k\le p-2$. It infers that
$k\not\equiv \epsilon_k \pmod {p-1}$, $\langle k\rangle=k$ and $v_p(k)=0$.
Hence only part (ii) happens. So we need only to show part (ii), i.e., to show
that
$$v_p(s(p,p-k))\ge (v_p(k)+1)\epsilon_k+1=\epsilon_k+1$$
with the equality holding if and only if
$v_p\big(B_{2\lfloor\frac{k}{2}\rfloor}\big)=0$.
But replacing $k$ by $p-k$ in Corollary \ref{cor2.10} tells that
for any integer $k$ with $2\le k\le p-2$, one has
$$v_p(s(p,p-k))\ge \epsilon_{p-k-1}+1=\epsilon_k+1$$
with the equality being true if and only if
$v_p\big(B_{p-1-2\lfloor\frac{p-k}{2}\rfloor}\big)=0$.
Since
\begin{align*}
p-1-2\Big\lfloor\frac{p-k}{2}\Big\rfloor
&=p-1-2\Big\lfloor\frac{p-1}{2}-\frac{k-1}{2}\Big\rfloor\\
&=p-1-2\Big(\frac{p-1}{2}+\Big\lfloor-\frac{k-1}{2}\Big\rfloor\Big)\\
&=-2\Big\lfloor-\frac{k-1}{2}\Big\rfloor=2\Big\lfloor\frac{k}{2}\Big\rfloor,
\end{align*}
it then follows that $v_p\big(B_{p-1-2\lfloor\frac{p-k}{2}\rfloor}\big)=0$
holds if and only if $v_p\big(B_{2\lfloor\frac{k}{2}\rfloor}\big)=0$.
Therefore Theorem \ref{thm1} is true when $a=1$. Now let
$2\le a\le p-1$. Assume that Theorem \ref{thm1} holds for the $a-1$ case.
In what follows, we show Theorem \ref{thm1} is true for the $a$ case.

We begin with the proof of part (iii).

{\bf(iii).} Let $a$ and $k$ be integers with $a\ge 4$ and $a(p-1)+2\le k\leq ap-2$.
By setting $t=ap-k$, one finds that showing the truth of part {\rm(iii)} is
equivalent to showing that
\begin{align}
v_p(s(ap,t))\ge a-t\label{4.1}
\end{align}
holds for any integer $t$ such that $2\le t\le a-2$.
We prove (\ref{4.1}) by induction on the integer $a\ge 4$.
First, let $a=4$. Then $t=2$, and so
\begin{align}\label{4.2}
s(ap,t)=s(4p,2)=(4p-1)!\sum_{i=1}^{4p-1}\frac{1}{i}=(4p-1)!H_{4p-1}.
\end{align}
Since $p\ge 5$, using Lemma \ref{lem6'}, one can easily
deduce that $v_p(H_{4p-1})\ge -1$. Then by $v_p((4p-1)!)=3$
together with (\ref{4.2}), we derive that $v_p(s(4p,2))\ge 2$
as (\ref{4.1}) expected. So (\ref{4.1}) is proved when $a=4$.

In what follows, we let $5\le a\le p-1$. Then $p>5$.
Assume that (\ref{4.1}) holds for the $a-1$ case.
Now we show that (\ref{4.1}) is true for the $a$ case.
From Lemma \ref{lem2}, we get that
\begin{align}
s(ap,t)=\sum_{i=1}^{t}s(p,i)s_p((a-1)p,t-i).\label{4.3}
\end{align}
For any integer $i$ with $1\le i\le t$, it follows from Lemma \ref{lem3} that
\begin{align}
s_p((a-1)p,t-i)=\sum_{j=t-i}^{(a-1)p}s((a-1)p,j)\binom{j}{j-t+i}p^{j-t+i}.\label{4.4}
\end{align}
Let $j$ be any integer such that $0\le j\le (a-1)p$. If we can show that
\begin{align}
v_p(s((a-1)p,j))\ge a-1-j, \label{4.5}
\end{align}
then from (\ref{4.4}) and (\ref{4.5}), one can derive that $v_p(s_p((a-1)p,t-i))\ge a-t$ since $i\ge 1$.
Using (\ref{4.3}), one then deduces the required inequality (\ref{4.1}).
So to finish the proof of part (iii), it remains to show the truth of (\ref{4.5}).
This will be done in what follows.

If $j=0$, then (\ref{4.5}) is obviously true since $s((a-1)p,0)=0$.
If $j=1$, then one has
$$
v_p(s((a-1)p,j))=v_p(s((a-1)p,1))=v_p(((a-1)p-1)!)=a-2=a-1-j.
$$
So (\ref{4.5}) holds when $j=1$. If $2\le j\le a-3$,
then it follows from the induction assumption
of (\ref{4.1}) that (\ref{4.5}) is true.
If $j=a-2$, then $(a-1)p+j$ is odd and so we can
deduce from Theorem \ref{thm2} that
$v_p(s((a-1)p,j))\ge 1=a-1-j$ .
If $j\ge a-1$, then (\ref{4.5}) is clearly true
since $v_p(s((a-1)p,j))\ge 0$. So (\ref{4.5}) holds
for all integers $j$ with $0\le j\le (a-1)p$.
Hence part {\rm(iii)} is proved.

Now we turn our attention to the proofs of parts (i) and (ii).
Let $k$ be an integer such that $2\le k\le a(p-1)+1$. Assume that
parts (i) and (ii) hold for all even integers $k$ with $2\le k\le a(p-1)$.
In the following, we show that parts (i) and (ii) are true for all odd
integers $k$ with $3\le k\le a(p-1)+1$. To do so, let $k$ be an odd
integer. Then $k\not\equiv 0\pmod {p-1}$ and $k-1$ is even with $2\le k-1\le a(p-1)$.

If $k\equiv 1 \pmod {p-1}$, then $k-1\equiv 0 \pmod {p-1}$.
Hence the truth of part (i) for the case of even number
gives us that $v_p(s(ap,ap-(k-1)))=0$.
Thus it follows from Theorem \ref{thm2} that
$$
v_p(s(ap,ap-k))=v_p(s(ap,ap-k+1))+v_p(k)+1=v_p(k)+1
$$
as expected. So part (i) is true for any odd integer $k$
with $k\equiv 1 \pmod {p-1}$.

If $k\not\equiv 1 \pmod {p-1}$, then $k-1\not\equiv 0 \pmod {p-1}$.
Hence the truth of part (ii) for the case of even number
tells us that $v_p(s(ap,ap-(k-1)))\ge 1$
with the equality holding if and only if $v_p(B_{\langle k-1\rangle})=0$.
Since $k$ is odd and one may write $k=\langle k\rangle+l(p-1)$ for some
nonnegative integer $l$, we can deduce that
$$
\langle k-1\rangle=\Big\langle 2\Big\lfloor\frac{k}{2}\Big\rfloor\Big\rangle
=\Big\langle 2\Big\lfloor\frac{\langle k\rangle+l(p-1)}{2}\Big\rfloor\Big\rangle
=\Big\langle 2\Big\lfloor\frac{\langle k\rangle}{2}\Big\rfloor+l(p-1)\Big\rangle
=\Big\langle 2\Big\lfloor\frac{\langle k\rangle}{2}\Big\rfloor\Big\rangle
=2\Big\lfloor\frac{\langle k\rangle}{2}\Big\rfloor.
$$
Thus $v_p(B_{\langle k-1\rangle})=0$ holds if and only if
$v_p\big(B_{2\big\lfloor\frac{\langle k\rangle}{2}\big\rfloor}\big)=0$.
From this together with Theorem \ref{thm2}, one then deduces that
$$
v_p(s(ap,ap-k))\ge v_p(k)+2
$$
with the equality being true if and only if
$v_p\big(B_{2\big\lfloor\frac{\langle k\rangle}{2}\big\rfloor}\big)=0$.
Hence part (ii) holds for any odd integer $k$ with $k\not\equiv 1 \pmod {p-1}$.

So to finish the proof of Theorem \ref{thm1}, it remains to show that parts
(i) and (ii) are true for all even integers $k$ with $2 \leq k \leq a(p-1)$.
This will be done in what follows.

In the remaining part of the proof, we always let $k$ be even
and $2 \leq k \leq a(p-1)$. Since $s(p,i)=0$ if
$i\ge p+1$ and $s_p((a-1)p,ap-k-i)=0$ if $i\le p-k-1$,
replacing $m$ by $p$, $n$ by $(a-1)p$ and $k$ by $ap-k$
in Lemma \ref{lem2} gives us that
\begin{align}
s(ap,ap-k)&=\sum_{i=1}^{ap-k}s(p,i)s_p((a-1)p,ap-k-i)\notag\\
&=\sum_{i=\max\{1,p-k\}}^{\min\{p,ap-k\}}s(p,i)s_p((a-1)p,ap-k-i).\label{4.6}
\end{align}
Now let $i$ be an integer such that $\max\{1,p-k\}\le i\le \min\{p,ap-k\}$.
Then $0\le ap-k-i\le (a-1)p$. By Corollary \ref{cor2.10}, we know that
$v_p(s(p,i))\ge 1$ if $2\le i\le p-1$. Moreover, write $ap-k-i=(a-1)p-(k+i-p)$,
where $0\le k+i-p\le \min\{k,(a-1)p\}$. From the inductive hypothesis of parts
(i) and (ii) together with the truth of part (iii) and Lemma \ref{lem7}, one can
deduce that
$$
v_p(s_p((a-1)p,ap-k-i))=v_p(s_p((a-1)p,(a-1)p-(k+i-p)))\ge 0
$$
with the equality holding if and only if $k+i-p\equiv 0 \pmod {p-1}$.
Define
$$
V_k:=\{i\in\mathbb{Z}: 2\le i\le p-1\ {\rm and} \ k+i-p\not\equiv 0\pmod {p-1}\}.
$$
Then $v_p(s(p,i))\ge 1$ and $v_p(s_p((a-1)p,ap-k-i))\ge 1$
if $i\in V_k$. This infers that for any $i\in V_k$, we have
\begin{align}
s(p,i)s_p((a-1)p,ap-k-i) \equiv 0 \pmod {p^2}.  \label{4.7}
\end{align}

Now we begin to prove part (i) for the case of even number $k$.

{\bf(i).} Let $k$ be an even integer such that
$2\le k\le a(p-1)$ and $k\equiv 0 \pmod {p-1}$.
Then one can write $k=l(p-1)$ for an integer $l$ with $1\le l\le a$.
We claim that if $k=l(p-1)$ for any integer $l$ with $1\le l\le a$, then
\begin{align}
s(ap,ap-k)\equiv \binom{a}{l}s(p, 1)^l \pmod{p^2}.\label{4.8}
\end{align}
Since $s(p, 1)=(p-1)!$, one has $v_p(s(p, 1))=v_p((p-1)!)=0$.
But $1\le l\le a\le p-1$, then from claim (\ref{4.8}), it follows that
$v_p(s(ap,ap-k))=0$ when $2\le k\le a(p-1)$ and $k\equiv 0 \pmod {p-1}$,
which arrives at the statement of part (i). Now we prove claim (\ref{4.8})
by induction on the integer $a$ with $2\le a \le p-1$.

First, let $a=2$. Then $l=1$ or $l=2$.
If $l=1$, then $k=p-1$. So $p-k=1$ and $2p-k=p+1$.
By (\ref{4.6}) and $s(p,p)=1$, we obtain that
\begin{align}\label{4.9}
s(2p,2p-(p-1))&=\sum_{i=1}^{p}s(p,i)s_p(p,p+1-i)\notag\\
&=s(p,1)+s_p(p,1)+\sum_{i=2}^{p-1}s(p,i)s_p(p,p+1-i).
\end{align}
For any integer $i$ with $2\le i\le p-1$, one has
$k+i-p=i-1\not\equiv 0 \pmod{p-1}$, and so $i\in V_{p-1}$.
Then by (\ref{4.7}), we derive that
\begin{align}\label{4.10}
\sum_{i=2}^{p-1}s(p,i)s_p(p,p+1-i)\equiv 0 \pmod {p^2}.
\end{align}
Using Lemma \ref{lem7}, we get that $s_p(p,1)\equiv s(p,1)\pmod{p^2}$.
Hence it follows from (\ref{4.9}) and (\ref{4.10}) that
$$
s(2p,2p-(p-1))\equiv 2s(p,1)=\binom{2}{1}s(p,1)\pmod{p^2}
$$
as (\ref{4.8}) claimed.
If $l=2$, then $k=2(p-1)$ and $2p-k=2$. Likewise, since
$s_p(p,1)\equiv s(p,1)\pmod{p^2}$ and $v_p(s(p,2))\ge 1$
and $v_p(s_p(p,0))=1$, one deduces from (\ref{4.6})
together with Lemmas \ref{lem5} and \ref{lem7} that
\begin{align*}
s(2p,2p-2(p-1))&=s(2p,2)=s(p,1)s_p(p,1)+s(p,2)s_p(p,0)\\
&\equiv s(p,1)^2=\binom{2}{2}s(p,1)^2\pmod{p^2}.
\end{align*}
Thus claim (\ref{4.8}) is true when $a=2$.

Now let $3\le a\le p-1$. Assume that claim (\ref{4.8}) holds for the
$a-1$ case. In what follows, we prove that claim (\ref{4.8}) is true
for the $a$ case. We divide this into the following three cases.

{\bf Case 1.} $l=1$. Then $k=p-1$. So $p-k=1$ and $ap-k-1=(a-1)p$.
From (\ref{4.6}) and $s_p((a-1)p,(a-1)p)=s(p,p)=1$, one derives that
\begin{align}\label{4.11}
s(ap,ap-k)=&\sum_{i=1}^{p}s(p,i)s_p((a-1)p,ap-k-i)\notag\\
=&s(p,1)+s_p((a-1)p,(a-1)p-k)+\sum_{i=2}^{p-1}s(p,i)s_p((a-1)p,ap-k-i).
\end{align}
Since $k=p-1$, by Lemma \ref{lem7} and the inductive
hypothesis of claim (\ref{4.8}), we get that
\begin{align}\label{4.12}
s_p((a-1)p,(a-1)p-k)&=s_p((a-1)p,(a-1)p-(p-1))\notag\\
&\equiv s((a-1)p,(a-1)p-(p-1))\notag\\
&\equiv \binom{a-1}{1}s(p,1)\pmod{p^2}.
\end{align}

For any integer $i$ with $2\le i\le p-1$, one has $1\leq k+i-p=i-1\leq p-2$
and so $k+i-p\not\equiv 0 \pmod {p-1}$, which implies that
$i\in V_{p-1}$. It then follows from (\ref{4.7}) that
\begin{align}\label{4.13}
\sum_{i=2}^{p-1}s(p,i)s_p((a-1)p,ap-k-i)\equiv 0 \pmod {p^2}.
\end{align}
Therefore by (\ref{4.11}) to (\ref{4.13}), we arrive at
\begin{align*}
s(ap, ap-k)\equiv s(p,1)+\binom{a-1}{1}s(p,1) = \binom{a}{1}s(p,1)\pmod{p^2}
\end{align*}
as (\ref{4.8}) asserted. Thus claim (\ref{4.8}) is proved
when $l=1$.

{\bf Case 2.} $2\leq l\leq a-1$. Then $2(p-1)\le k=l(p-1)\le (a-1)(p-1)$.
So $p-k<0$ and $p+1\leq ap-k \leq(a-1)p-1$.
From (\ref{4.6}) and $s(p,p)=1$, one derives that
\begin{align}\label{4.14}
s(ap,ap-k)
=&\sum_{i=1}^{p}s(p,i)s_p((a-1)p,ap-k-i)\notag\\
=&s(p,1)s_p((a-1)p,ap-k-1)+s_p((a-1)p,(a-1)p-k)\notag\\
&+\sum_{i=2}^{p-1}s(p,i)s_p((a-1)p,ap-k-i).
\end{align}

Since $k=l(p-1)$ is even and $2\le l\le a-1$, by using Lemma \ref{lem7}
together with the inductive hypothesis of claim (\ref{4.8}),
we deduce that
\begin{align}\label{4.15}
s_p((a-1)p,ap-k-1)
&=s_p((a-1)p,(a-1)p-(k-(p-1)))\notag\\
&\equiv s((a-1)p,(a-1)p-(k-(p-1)))\notag\\
&\equiv s((a-1)p,(a-1)p-(l-1)(p-1))\notag\\
&\equiv  \binom{a-1}{l-1}s(p,1)^{l-1}\pmod{p^2}
\end{align}
and
\begin{align}\label{4.16}
s_p((a-1)p,(a-1)p-k)&\equiv s((a-1)p,(a-1)p-k)\notag\\
&\equiv s((a-1)p,(a-1)p-l(p-1))\notag\\
&\equiv \binom{a-1}{l}s(p,1)^{l}\pmod{p^2}.
\end{align}

For any integer $i$ with $2\le i\le p-1$, one has
$(l-1)(p-1)+1\leq k+i-p\leq l(p-1)-1$
and so $k+i-p\not\equiv 0 \pmod {p-1}$. It infers that
$i\in V_k$. So from (\ref{4.7}), one obtains that
\begin{align} \label{4.17}
\sum_{i=2}^{p-1}s(p,i)s_p((a-1)p,ap-k-i)\equiv 0 \pmod {p^2}.
\end{align}
Thus by (\ref{4.14}) to (\ref{4.17}), we derive that
\begin{align*}
s(ap,ap-k)\equiv s(p,1)\binom{a-1}{l-1}s(p,1)^{l-1}+
\binom{a-1}{l}s(p,1)^l \equiv \binom{a}{l}s(p,1)^l\pmod{p^2}
\end{align*}
as (\ref{4.8}) asserted. Thus claim (\ref{4.8}) is proved
when $2\leq l\leq a-1$.

{\bf Case 3.} $l=a$. Then $k=a(p-1)$, $p-k< 0$ and
$3\leq ap-k=a\leq p-1$. So by (\ref{4.6}), one deduces that
\begin{align}\label{4.18}
s(ap,ap-k)&=s(ap,a)=\sum_{i=1}^as(p,i)s_p((a-1)p,a-i)\notag\\
&=s(p,1)s_p((a-1)p,a-1)+\sum_{i=2}^as(p,i)s_p((a-1)p,a-i).
\end{align}
Using Lemma \ref{lem7} and the inductive hypothesis
of claim (\ref{4.8}), we obtain that
\begin{align}\label{4.19}
s_p((a-1)p,a-1)&=s_p((a-1)p,(a-1)p-(a-1)(p-1))\notag\\
&\equiv s((a-1)p,(a-1)p-(a-1)(p-1))\notag\\
&\equiv  s(p,1)^{a-1} \pmod{p^2}.
\end{align}
For any integer $i$ with $2\leq i\leq a\le p-1$, one has
$k+i-p=a(p-1)-(p-1)+i-1\equiv i-1\not\equiv 0\pmod {p-1}$,
and so $i\in V_k$. Hence by (\ref{4.7}), one gets that
\begin{align} \label{4.20}
\sum_{i=2}^{a}s(p,i)s_p((a-1)p,ap-k-i)\equiv 0 \pmod {p^2}.
\end{align}
It then follows from (\ref{4.18}) to (\ref{4.20}) that
\begin{align*}
s(ap,ap-k)=s(ap,a)\equiv s(p,1)s_p((a-1)p,a-1)\equiv s(p,1)^a\pmod{p^2}.
\end{align*}
Namely, claim (\ref{4.8}) is true when $l=a$.
This completes the proof of part (i).

Finally, we prove part (ii) for the case of even number $k$.

{\bf(ii).} Let $k$ be an even integer such that $2\le k\le a(p-1)$ and
$k\not\equiv 0\pmod{p-1}$. Then there exists a unique integer $l$ with
$1\le l\le a$ such that $(l-1)(p-1)+2\leq k\leq l(p-1)-2$. Thus
$\langle k\rangle=k-(l-1)(p-1)$.
We claim that if $(l-1)(p-1)+2\leq k\leq l(p-1)-2$ for
any integer $l$ with $1\le l\le a$, then
\begin{align}\label{4.21}
s(ap,ap-k)\equiv a\binom{a-1}{l-1}s(p,1)^{l-1}
s(p,p-\langle k\rangle)\pmod{p^2}.
\end{align}
Since $3\le p-\langle k\rangle\le p-2$ and $\langle k\rangle$ is even,
by Corollary \ref{cor2.10}, one gets that $v_p(s(p,p-\langle k\rangle))\ge 1$
with the equality holding if and only if $v_p(B_{\langle k\rangle})=0$.
Thus for any even integer $k$ with $2\le k\le a(p-1)$, it
follows from claim (\ref{4.21}) that $v_p(s(ap,ap-k))\ge 1$
if $k\not\equiv0 \pmod {p-1}$ with the equality being true
if and only if $v_p(B_{\langle k\rangle})=0$. So to finish the
proof of part (ii), it remains to show the truth of claim (\ref{4.21}).
We proceed this with induction on the integer $a$ with $2\le a \le p-1$.

For the case $a=2$, one has $l=1$ or $l=2$. If $l=1$, then
$2\leq k\leq p-3$. So $\langle k\rangle=k$, $3\le p-k\le p-2$
and $p+3\leq 2p-k\leq 2p-2$. Since $s_p(p,p)=s(p,p)=1$,
it follows from (\ref{4.6}) that
\begin{align}\label{4.22}
s(2p,2p-k)&=\sum_{i=p-k}^{p}s(p,i) s_p(p,2p-k-i)\notag\\
&=s(p,p-k)+s_p(p,p-k)+\sum_{i=p-k+1}^{p-1}s(p,i) s_p(p,2p-k-i).
\end{align}
For any integer $i$ with $4\le p-k+1\le i\le p-1$,
we have $1\le k+i-p\le p-4$. It implies that $i\in V_k$.
Hence (\ref{4.7}) tells us that $s(p,i) s_p(p,2p-k-i)\equiv 0\pmod {p^2}$.
So it follows that
\begin{align}\label{4.23}
\sum_{i=p-k+1}^{p-1}s(p,i) s_p(p,2p-k-i)\equiv 0 \pmod {p^2}.
\end{align}
Since $k$ is even, by Lemma \ref{lem7}, we know that
$s_p(p,p-k)\equiv s(p,p-k)\pmod{p^2}$. Thus (\ref{4.22})
and (\ref{4.23}) give us that
\begin{align*}
s(2p,2p-k)\equiv 2s(p,p-k)\equiv 2\binom{1}{0}s(p,p-\langle k\rangle)\pmod{p^2}
\end{align*}
as (\ref{4.21}) claimed. Hence claim (\ref{4.21}) is true when $a=2$ and $l=1$.

If $l=2$, then $p+1\leq k\leq 2(p-1)-2$.
So one has $\langle k\rangle=k-(p-1)$, $p-k< 0$ and
$4\leq 2p-k\leq p-1$. Thus by (\ref{4.6}), we get that
\begin{align}\label{4.24}
s(2p,2p-k)=&\sum_{i=1}^{2p-k}s(p,i) s_p(p,2p-k-i)\notag\\
=&s(p,1) s_p(p,2p-k-1)+\sum_{i=2}^{2p-k-2}s(p,i) s_p(p,2p-k-i)\notag\\
&+s(p,2p-k-1) s_p(p,1)+s(p,2p-k) s_p(p,0).
\end{align}
Note that $2p-k-1=p-(k-(p-1))=p-\langle k\rangle$ and
$\langle k\rangle$ is even.  So by Lemma \ref{lem7},
one deduces that $s_p(p,1)\equiv s(p,1)\pmod{p^2}$ and
\begin{align}\label{4.25}
s_p(p,2p-k-1)=s_p(p,p-\langle k\rangle)\equiv
s(p,p-\langle k\rangle) \pmod{p^2}.
\end{align}
If $2\le i\le 2p-k-2\le p-3$, then $3\le k+i-p\le p-2$.
It infers that $i\in V_k$. So from (\ref{4.7}),
one obtains that
\begin{align} \label{4.26}
\sum_{i=2}^{2p-k-2}s(p,i) s_p(p,2p-k-i)\equiv 0 \pmod {p^2}.
\end{align}
Since $4\le 2p-k\le p-1$, one has $v_p(s(p,2p-k))\ge 1$.
It then follows from the fact $v_p(s_p(p,0))=1$ that
$s(p,2p-k)s_p(p,0)\equiv 0\pmod{p^2}$. Thus by (\ref{4.24})
to (\ref{4.26}), we arrive at
\begin{align*}
s(2p,2p-k)\equiv &2s(p,1)s(p,p-(k-(p-1)))\\
\equiv& 2\binom{1}{1}s(p,1)s(p,p-\langle k\rangle)\pmod{p^2}.
\end{align*}
This completes the proof of claim (\ref{4.21}) when $a=2$.

In what follows, we let $3\le a\le p-1$. Assume that claim
(\ref{4.21}) holds for the $a-1$ case. To show that
claim (\ref{4.21}) is true for the $a$ case, we consider
the following three cases.

{\bf Case 1.} $l=1$. Then $2\leq k\leq p-3$. So $\langle k\rangle=k$,
$p-k\ge 3$ and $(a-1)p+3\leq ap-k\leq ap-2$. Since
$s_p((a-1)p,(a-1)p)=s(p,p)=1$, by (\ref{4.6}) we get that
\begin{align}\label{4.27}
&s(ap,ap-k)\notag\\
= &\sum_{i=p-k}^ps(p,i)s_p((a-1)p,ap-k-i)\notag\\
= &s(p,p-k)+s_p((a-1)p,(a-1)p-k)
+\sum_{i=p-k+1}^{p-1}s(p,i)s_p((a-1)p,ap-k-i).
\end{align}
Note that $k$ is even and $2\le k\le p-3\le (a-1)p-1$. So by using Lemma \ref{lem7}
and the inductive hypothesis of claim (\ref{4.21}), one derives that
\begin{align}\label{4.28}
s_p((a-1)p,(a-1)p-k)\equiv s((a-1)p,(a-1)p-k)
\equiv (a-1)s(p,p-k)\pmod{p^2}.
\end{align}
For any integer $i$ with $4\le p-k+1\leq i\leq p-1$, we have
$1\le k+i-p\le p-4$, and so $i\in V_k$. By (\ref{4.7}),
one gets that
\begin{align}\label{4.29}
\sum_{i=p-k+1}^{p-1}s(p,i) s_p(p,2p-k-i)\equiv 0 \pmod {p^2}.
\end{align}
Hence it follows from (\ref{4.27}) to (\ref{4.29}) that
\begin{align*}
s(ap,ap-k)&\equiv s(p,p-k)+(a-1)s(p,p-k)\\
&\equiv as(p,p-k)\equiv a\binom{a-1}{0}s(p,p-\langle k\rangle)\pmod{p^2}
\end{align*}
as (\ref{4.21}) claimed. Hence claim (\ref{4.21}) is proved when $l=1$.

{\bf Case 2.} $2\leq l\leq a-1$. Then $(l-1)(p-1)+2\le k\le l(p-1)-2$.
Hence one has $\langle k\rangle=k-(l-1)(p-1)$, $p-k< 0$ and
$p+4\leq ap-k\leq ap-3$. From (\ref{4.6}) and $s(p,p)=1$ we know that
\begin{align} \label{4.30}
s(ap,ap-k)=&\sum_{i=1}^{p}s(p,i)s_p((a-1)p,ap-k-i)\notag\\
=&s(p,1)s_p((a-1)p,ap-k-1)+s_p((a-1)p,(a-1)p-k)\notag\\
&+\sum_{i=2}^{p-1}s(p,i)s_p((a-1)p,ap-k-i).
\end{align}

Notice that $ap-k-1=(a-1)p-(k-(p-1))$ and $(l-1)(p-1)+2\le k\le l(p-1)-2$
with $2\le l\le a-1$. So $2\le (l-2)(p-1)+2\le k-(p-1)\le (l-1)(p-1)-2$
and $\langle k-(p-1)\rangle=\langle k\rangle$.  Thus using Lemma \ref{lem7}
and the induction assumption of claim (\ref{4.21}), we obtain that
\begin{align}\label{4.31}
s_p((a-1)p,ap-k-1)=&s_p((a-1)p,(a-1)p-(k-(p-1)))\notag\\
\equiv &s((a-1)p,(a-1)p-(k-(p-1)))\notag\\
\equiv &(a-1)\binom{a-2}{l-2}s(p,1)^{l-2}
s(p,p-\langle k\rangle)\pmod{p^2}
\end{align}
and
\begin{align}\label{4.32}
s_p((a-1)p,(a-1)p-k)\equiv & s((a-1)p,(a-1)p-k)\notag\\
\equiv & (a-1)\binom{a-2}{l-1}s(p,1)^{l-1}
s(p,p-\langle k\rangle)\pmod{p^2}.
\end{align}

For any integer $i$ with $2\leq i\leq p-1$, one has
$(l-2)(p-1)+3\leq k+i-p\leq l(p-1)-3$.
If $k+i-p\equiv 0 \pmod {p-1}$, then $k+i-p=(l-1)(p-1)$.
Thus $i=p-(k-(l-1)(p-1))=p-\langle k\rangle$ and
$ap-k-i=(a-1)p-(l-1)(p-1)$. If $k+i-p\not\equiv 0 \pmod {p-1}$,
then $i\ne p-\langle k\rangle$ and $i\in V_k$.
It implies that $s(p,i)s_p((a-1)p,ap-k-i) \equiv 0\pmod{p^2}$
when $2\le i\le p-1$ and $i\ne p-\langle k\rangle$.
Hence one gets that
\begin{align}\label{4.33}
&\sum_{i=2}^{p-1}s(p,i)s_p((a-1)p,ap-k-i)\notag\\
\equiv &s(p,p-\langle k\rangle)s_p((a-1)p,(a-1)p-(l-1)(p-1)) \pmod {p^2}.
\end{align}
Since $2\le l\le a-1$ and $(l-1)(p-1)$ is even, it follows from the
truth of Lemma \ref{lem7} and claim (\ref{4.8}) that
\begin{align}\label{4.34}
&s_p((a-1)p,(a-1)p-(l-1)(p-1))\notag\\
\equiv &s((a-1)p,(a-1)p-(l-1)(p-1))\equiv
\binom{a-1}{l-1}s(p,1)^{l-1}\pmod{p^2}.
\end{align}
Therefore by (\ref{4.30}) to (\ref{4.34}), we conclude that
\begin{align*}
&s(ap,ap-k)\\
\equiv &\Big((a-1)\binom{a-2}{l-2}+(a-1)\binom{a-2}{l-1}
+\binom{a-1}{l-1}\Big)
s(p,1)^{l-1}s(p,p-\langle k\rangle)\\
\equiv&a\binom{a-1}{l-1}s(p,1)^{l-1}s(p,p-\langle k\rangle)\pmod{p^2}
\end{align*}
as (\ref{4.21}) asserted. Thus claim (\ref{4.21}) is proved
when $2\le l\le a-1$.

{\bf Case 3.} $l=a$. Then $(a-1)(p-1)+2\leq k\leq a(p-1)-2$.
So $\langle k\rangle=k-(a-1)(p-1)$, $p-k<0$ and
$5\le a+2\leq ap-k\leq p+a-3\leq 2p-4$. From (\ref{4.6}), one gets that
\begin{align}\label{4.35}
&s(ap,ap-k)\notag\\
=&\sum_{i=1}^{\min\{p,ap-k\}}s(p,i)s_p((a-1)p,ap-k-i)\notag\\
=&s(p,1)s_p((a-1)p,ap-k-1)+\sum_{i\in W_k}s(p,i)s_p((a-1)p,ap-k-i),
\end{align}
where $W_k:=\{i\in \mathbb{Z}: 2\le i\le \min\{p,ap-k\}\}$.

Since $ap-k-1=(a-1)p-(k-(p-1))$ and $(a-2)(p-1)+2\le
k-(p-1)\le (a-1)(p-1)-2$, it follows from Lemma \ref{lem7}
and the inductive hypothesis of claim (\ref{4.21}) that
\begin{align}\label{4.36}
s_p((a-1)p,ap-k-1)&=s_p((a-1)p,(a-1)p-(k-(p-1)))\notag\\
&\equiv s((a-1)p,(a-1)p-(k-(p-1)))\notag\\
&\equiv (a-1)s(p,1)^{a-2}s(p,p-\langle k\rangle)\pmod{p^2}.
\end{align}

Since $2\le \langle k\rangle=k-(a-1)(p-1)\le p-3$ and $a\ge 3$, one has
$2\le p-\langle k\rangle\le p-3$ and $p-\langle k\rangle\le ap-k-2$.
It implies that $p-\langle k\rangle \in W_k$.
For the integer $i$ with $i=p-\langle k\rangle$, we have
$ap-k-i=ap-k-(p-\langle k\rangle)=(a-1)p-(a-1)(p-1)$.
So using Lemma \ref{lem7} and claim (\ref{4.8}), one derives that
\begin{align}\label{4.37}
s_p((a-1)p,ap-k-(p-\langle k\rangle))&=s_p((a-1)p,(a-1)p-(a-1)(p-1))\notag\\
&\equiv s((a-1)p,(a-1)p-(a-1)(p-1))\notag\\
&\equiv s(p,1)^{a-1}\pmod{p^2}.
\end{align}

If $a+2\leq ap-k\leq p-1$, then $W_k=\{i\in \mathbb{Z}: 2\le i\le ap-k\}$.
For any integer $i$ with $i\in W_k$, one has
$(a-2)(p-1)+3\leq k+i-p\leq a(p-1)-3$. So if $k+i-p\equiv 0 \pmod {p-1}$,
then we must have $k+i-p=(a-1)(p-1)$, i.e.,
$i=p-(k-(a-1)(p-1))=p-\langle k\rangle$.
If $k+i-p\not\equiv 0 \pmod {p-1}$, then
$i\ne p-\langle k\rangle$ and $i\in V_k$, which infers that
$s(p,i)s_p((a-1)p,ap-k-i)\equiv 0\pmod{p^2}$.
Hence one obtains that
\begin{align}\label{4.38}
\sum_{i\in W_k}s(p,i)s_p((a-1)p,ap-k-i)
\equiv &s(p,p-\langle k\rangle)s_p((a-1)p,ap-k-(p-\langle k\rangle))\notag\\
\equiv &s(p,1)^{a-1}s(p,p-\langle k\rangle) \pmod {p^2}.
\end{align}

If $p\leq ap-k\leq p+a-3$, then $W_k=\{i\in \mathbb{Z}: 2\le i\le p\}$
and $(a-1)(p-1)+2\leq k\leq (a-1)p$. Since $k$ is even, one has
$(a-1)(p-1)+2\le k\le (a-1)p-2$ or $k=(a-1)p$. So by the truth of
part (iii) and $s((a-1)p,0)=0$, we know that
\begin{align*}
v_p(s((a-1)p,(a-1)p-k))\ge a-1+k-(a-1)p\ge 2.
\end{align*}
It then follows from Lemma \ref{lem7} that
\begin{align}\label{4.39}
s_p((a-1)p,(a-1)p-k)\equiv s((a-1)p,(a-1)p-k)\equiv 0 \pmod{p^2}.
\end{align}
Likewise, for any integer $i$ with $i\in W_k$, one has
$k+i-p\equiv 0 \pmod {p-1}$ if and only if $i=p-\langle k\rangle$.
Let $U_k:=\{i\in W_k: i\ne p-\langle k\rangle
\ {\rm and}\ i\ne p\}$. Then $\emptyset \ne U_k\subseteq V_k$.
Hence by (\ref{4.7}) together with (\ref{4.37}) and (\ref{4.39}), one gets that
\begin{align}\label{4.40}
&\sum_{i\in W_k}s(p,i)s_p((a-1)p,ap-k-i)\notag\\
=&\sum_{i\in U_k}s(p,i)s_p((a-1)p,ap-k-i)+s(p,p)s_p((a-1)p,(a-1)p-k)\notag\\
&+s(p,p-\langle k\rangle)s_p((a-1)p,ap-k-(p-\langle k\rangle))\notag\\
\equiv &s(p,p-\langle k\rangle)s_p((a-1)p,ap-k-(p-\langle k\rangle))\notag\\
\equiv &s(p,1)^{a-1}s(p,p-\langle k\rangle) \pmod {p^2}.
\end{align}
Therefore (\ref{4.35}) and (\ref{4.36}) together with (\ref{4.38})
and (\ref{4.40}) give us that
\begin{align*}
s(ap,ap-k)&\equiv (a-1)s(p,1)^{a-1}s(p,p-\langle k\rangle)+
s(p,p-\langle k\rangle)s(p,1)^{a-1}\\
&\equiv as(p,1)^{a-1}s(p,p-\langle k\rangle)\notag\\
&\equiv a\binom{a-1}{a-1}s(p,1)^{a-1}s(p,p-\langle k\rangle) \pmod{p^2}
\end{align*}
as desired. The proof of claim (\ref{4.21}) is complete.
So part {\rm(ii)} is proved.

This finishes the proof of Theorem \ref{thm1}. \qed

\section{Concluding remarks}

In \cite{[QH]}, we gave a formula for $v_2(s(2^n, k))$ with $k$
being an integer such that $1\le k\le 2^n$. In \cite{[QFH]}, Qiu,
Feng and Hong presented a formula for $v_3(s(a3^n, k))$ with $k$
being an integer such that $1\le k\le a3^n$, where $a\in\{1, 2\}$.
In this paper, we arrive at an exact expression or a lower bound
of $v_p(s(ap, k))$ with $a$ and $k$ being
integers such that $1\le a\le p-1$ and $1\le k\le ap$. It is
natural to consider the $p$-adic valuation of the Stirling
number $s(ap^n, k)$, where $a, n$ and $k$ being
integers such that $1\le a\le p-1, n\ge 2$ and $1\le k\le ap^n$.
For any odd prime $p$ and any positive integer $k$,
recall that $\epsilon_k$ is defined by $\epsilon_k :=0$
if $k$ is even and $\epsilon_k :=1$ if $k$ is odd,
and $\langle k\rangle$ denotes the
integer such that $0\le \langle k\rangle\le p-2$
and $k\equiv \langle k\rangle \pmod {p-1}$. We propose the
following conjecture.

\begin{cnj}\label{cnj1}
Let $p$ be an odd prime. Let $a,n,m,k$ be positive integers such that
$1\le a\leq p-1$, $1\le m\le n$ and $2\le k\le ap^n-2$.
Then each of the following is true:

{\rm (i).} If $2\le k \le a(p-1)p^{m-1}+1<ap^m$, then
\begin{align*}
v_p(s(ap^n, ap^m-k))=\frac{a}{p-1}(p^n-p^m)-(n-m)(ap^m-k)
+m+(m+v_p(k))\epsilon_k+T_k,
\end{align*}
where
$$
T_k:=\Big\{\begin{array}{ll}
 -1- v_p(\lfloor\frac{k}{2}\rfloor),
 & {\it if}\ k\equiv \epsilon_k\pmod{p-1}; \\
v_p(B_{2\lfloor\frac{\langle k\rangle}{2}\rfloor}),
& {\it if}\ k\not\equiv \epsilon_k\pmod{p-1}.
\end{array}
$$

{\rm (ii).} If $a\ge 4$ and $a(p-1)+2\le k\le ap-2$, then
\begin{align*}
v_p(s(ap^n,ap-k))\ge\frac{a}{p-1}(p^n-p)-(n-1)(ap-k)+a+k-ap.
\end{align*}
\end{cnj}

From Theorem \ref{thm1}, we can see that for all primes
$p\ge 5$ part (ii) of Conjecture \ref{cnj1} is true when $n=1$ and part (i)
of Conjecture \ref{cnj1} also holds for $n=1$ and
$k\equiv \epsilon_k\pmod{p-1}$. By the main result in
\cite{[QFH]}, we know that Conjecture
\ref{cnj1} is true when $p=3$.

Letting $m=n$, Conjecture \ref{cnj1} becomes the following conjecture.

\begin{cnj}\label{cnj2}
Let $p$ be an odd prime. Let $a,n,k$ be positive integers such that
$1\le a\leq p-1$ and $2\le k\le a(p-1)p^{n-1}+1$.
Then
\begin{align*}
 v_p(s(ap^n, ap^n-k)) =\Big\{\begin{array}{ll}
  n+(n+v_p(k))\epsilon_k
 -1- v_p(\lfloor\frac{k}{2}\rfloor),
 & {\it if}\ k\equiv \epsilon_k\pmod{p-1}; \\
n+(n+v_p(k))\epsilon_k
+v_p(B_{2\lfloor\frac{\langle k\rangle}{2}\rfloor}),
& {\it if}\ k\not\equiv \epsilon_k\pmod{p-1}.
\end{array}
\end{align*}
\end{cnj}

On the other hand, Corollary 4 in \cite{[KY]} gives us that
$$
v_p(s(ap^n, ap^m))=\frac{a}{p-1}(p^n-p^m)-a(n-m)p^n.
$$
So we suggest the following conjecture as the conclusion of this paper.

\begin{cnj}\label{cnj3}
Let $p$ be a prime. Let $a,n,m$ and $k$ be positive integers such that
$1\le a\leq p-1$, $1\le m\le n$ and $2\le k\le a(p-1)p^{m-1}+1<ap^m$.
Then
\begin{align*}
v_p(s(ap^n, ap^m-k))
=v_p(s(ap^n, ap^m))+v_p(s(ap^n, ap^n-k))
+\big(2\Big\lfloor\frac{k}{2}\Big\rfloor-1\big)(n-m).
\end{align*}
\end{cnj}

\begin{center}
{\bf Acknowledgement}
\end{center}
The authors would like to thank the anonymous referee
for careful reading of the manuscript and helpful
comments and suggestions.

\bibliographystyle{amsplain}

\end{document}